\documentclass[12pt, onecolumn]{article}

\usepackage[dvipdfmx]{graphicx}
\usepackage[utf8]{inputenc}
\usepackage[T1]{fontenc}
\usepackage[numbers]{natbib}
\usepackage{amsmath}
\usepackage{amsthm}
\usepackage{amssymb}
\usepackage{algorithm}
\usepackage{algorithmic}
\usepackage{ecltree}
\usepackage{enumerate}
\usepackage{xcolor}
\usepackage{comment}
\usepackage{tikz}
\usepackage{wrapfig}
\usepackage{mathrsfs}
\usepackage{mathtools}
\usepackage[margin=25truemm]{geometry}
\mathtoolsset{showonlyrefs=true}

\usepackage{bm}
\usepackage{bbm}
\usepackage{float}
\usepackage{cases}
\usepackage{longtable}
\usepackage{diagbox}

\newtheorem{proposition}{Proposition}
\newtheorem{lemma}[proposition]{Lemma}
\newtheorem{theorem}[proposition]{Theorem}
\newtheorem{corollary}[proposition]{Corollary}

\theoremstyle{definition}
\newtheorem{definition}[proposition]{Definition}

\newcommand{\NN}{{\mathcal{N}}}
\newcommand{\PP}{{\mathcal{P}}}

\newcommand{\Reach}{R}
\newcommand{\reach}{\tilde{R}}

\newcommand{\eqlab}[2]{\overset{(\mathrm{#1})}{#2}}

\makeatletter

\@addtoreset{equation}{section}
\makeatother

\title{The Computational Complexity of\\
the Ungar Games on Distributive Lattices}
\author{Kengo Hashimoto}
\date{}

\begin{document}
\maketitle

\begin{abstract}
The Ungar game, introduced by Defant et al., is a combinatorial game played on a finite lattice,
which originates from a geometric transformation introduced by Ungar to resolve Scott's slope problem.
When played on a distributive lattice $L$, the Ungar game reduces to repeatedly choosing and removing a non-empty set of maximal elements from a finite poset $P$, where $L$ is represented as the lattice $J(P)$ of the order ideals of $P$ by Birkhoff's representation theorem.
Defant et al.\ raised the open question of whether determining the outcome of the Ungar game on $J(P)$ is \textsf{PSPACE}-complete with respect to $|P|$.
In this paper, we settle this open question in the affirmative.
More precisely, we prove that the problem is \textsf{PSPACE}-complete even when restricted to posets $P$ with maximum chain length $l(P) \le 3$ such that the underlying graph of the Hasse diagram of $P$ is bipartite, where the length of a chain is defined as the number of elements in the chain minus one.
Furthermore, we show that the problem is \textsf{NP}-complete for $l(P) \le 2$, whereas the case $l(P) \le 1$ is known to belong to \textsf{LOGSPACE}.
\end{abstract}

\section{Introduction}
\label{sec:intro}

\subsection{Background} 
The \emph{Ungar games}, introduced by Defant et al.\ \cite{DKW24}, are combinatorial games (two-player games without chance elements or hidden information such as \textsc{Chess} and \textsc{Checkers}) played on a finite lattice $L$.
In the Ungar game on $L$, starting with $x$ set to the maximum element $1 \in L$, two players alternately perform a non-trivial \emph{Ungar move} \cite{DL23},
which replaces the element $x$ with $\bigwedge (\{x\} \cup T)$ for some non-empty set $T$ of elements covered by $x$.
The player who cannot make a legal move loses.\footnote{Ungar moves and Ungar games can be naturally extended to meet-semilattices as in \cite{DL23}.}

In the case where $L$ is a distributive lattice, $L$ is isomorphic to the lattice $J(P)$ of order ideals of a finite poset $P$ by Birkhoff's representation theorem, and the Ungar game on $L = J(P)$ is viewed as the following combinatorial game on $P$.
\begin{quote}
The game is played on the poset $P$.
In each move, a player chooses a non-empty $S \subseteq \max(P)$ and removes all the elements in $S$ from $P$
(i.e., replaces $P$ with $P \setminus S$),
where $\max(P)$ denotes the set of all maximal elements in $P$.
The player who cannot make a move in a turn because $P = \emptyset$ loses.
\end{quote}

Regarding the \emph{outcome} (i.e., which player has a winning strategy) of the Ungar game on $J(P)$,
Defant et al.\ \cite{DKW24} characterized the winning positions for the case $P = \mathbb{N}^2$ (i.e., $J(P)$ is Young's lattice).
Choi and Gan \cite{CG24} determined the winning positions for the case where $P$ is a shift staircase, settling an open conjecture posed in \cite{DKW24}. 
Paltrowitz \cite{Pal26} proved that the problem of determining the outcome of the Ungar game on $J(P)$ for a graded poset $P$ can be reduced to an evaluation of Boolean NAND logic circuits by extracting a forest structure called the \emph{skeleton} from $P$, establishing that the problem belongs to \textsf{LOGSPACE} and generalizing the result in \cite{DKW24} on $P = \mathbb{N}^2$ to
the higher-dimensional case $P = \mathbb{N}^d$.
On the other hand, Defant et al.\ \cite{DKW24} also asked the following question on the computational complexity of determining the outcome in the Ungar game on $J(P)$, which has remained open:
\begin{quote}
Are Ungar games on distributive lattices $J(P)$ \textsf{PSPACE}-complete in $|P|$?
\end{quote}

\subsection{Contribution and Organization} 
In this paper, we answer the above open question in the affirmative by showing that determining the outcome of the Ungar game on $J(P)$ is \textsf{PSPACE}-complete.
More precisely, we prove stronger results as Theorems \ref{thm:pspace} and \ref{thm:np} below, where $l(P)$ denotes the length (i.e., the number of elements minus one) of the longest chain in a poset $P$.

\newtheorem*{Thm:pspace}{\rm\bf Theorem~\ref{thm:pspace}}
\begin{Thm:pspace}
The problem of determining the outcome of the Ungar game on $J(P)$ for a given finite poset $P$ is \textsf{PSPACE}-complete even if
we restrict $P$ to the cases where $l(P) \leq 3$ and the underlying graph of the Hasse diagram of $P$ is bipartite.
\end{Thm:pspace}
\newtheorem*{Thm:np}{\rm\bf Theorem~\ref{thm:np}}
\begin{Thm:np}
The problem of determining the outcome of the Ungar game on $J(P)$ for a given finite poset $P$ with $l(P) \leq 2$ is \textsf{NP}-complete.
\end{Thm:np}

Note that in the case of $l(P) \leq 1$, the poset $P$ is a graded poset and thus the problem belongs to \textsf{LOGSPACE} as shown in \cite{Pal26}.

The remainder of this paper is organized as follows. 
In the rest of this section, we review additional related works not mentioned above.
In Section~\ref{sec:preliminaries}, we introduce necessary notations and definitions and establish several lemmas required to prove our main results. 
In Section~\ref{sec:main}, we present our main theorems: we prove Theorem~\ref{thm:pspace} in Section~\ref{subsec:pspace} and Theorem~\ref{thm:np} in Section~\ref{subsec:np}.

\subsection{Related Works} 
The conceptual origin of the Ungar games dates back to a geometric transformation introduced by Ungar \cite{Ung82} to resolve Scott's slope problem \cite{Sco70} on the minimum number of slopes determined by points in the plane that do not all lie on a single line.
This operation was extended to lattices and termed the Ungar move by Defant and Li \cite{DL23}.
They introduced \emph{Ungarian Markov chains}, which perform random Ungar moves on a finite lattice, and analyzed the expected number of steps needed to transition from the maximum element to the minimum element.

Regarding the Ungar games on non-distributive lattices,
Defant et al.\ \cite{DKW24} established several enumerative results regarding the number of winning positions for well-known families of lattices,
and proved that the problem of determining the outcome of the Ungar game on a lattice $L$ is $\mathsf{NC^1}$-hard with respect to $|L|$.
Furthermore, Choi and Gan \cite{CG24} also characterized the winning positions for the Ungar games on Young-Fibonacci lattices, thereby settling another open question posed in \cite{DKW24}.

Similar to the Ungar game on $J(P)$, \emph{poset games} are another well-known class of combinatorial games played on a poset.
A poset game is played on a finite poset $P$. In each move, a player chooses an arbitrary element $x \in P$ and removes $x$ together with all elements greater than $x$ from $P$.
The player who cannot make a legal move loses.
The problem of determining the outcome of the poset game on a given poset is known to be \textsf{PSPACE}-complete \cite{FGG22}.

\section{Preliminaries}
\label{sec:preliminaries}

For the sake of presentation in this paper,
we reinterpret the Ungar game on $J(P)$ as a combinatorial game played on a directed acyclic graph (DAG)
by viewing the removal of maximal elements of $P$ as the removal of nodes with in-degree $0$ in the Hasse diagram of $P$ as follows.
\begin{quote}
Initially, there is a finite DAG $G = (V, A)$.
A player's move is to choose a non-empty $S \subseteq L_0(V)$ and remove $S$ from $V$,
where $L_0(V)$ is the set of all nodes $x \in V$ with in-degree $0$.
Namely, $V$ is replaced with the subgraph of $G$ induced by $V \setminus S$ as the result of the move.
The player who cannot make a move in a turn because $V = \emptyset$ loses.
\end{quote}

We refer to the above game as \emph{the Ungar game on the DAG $G$ (or $V$)}.
Without loss of generality, we assume the DAG has no parallel arcs.
In a game with initial position $V$, 
each position is identified with the subset $U \subseteq V$ of the remaining nodes in the DAG, 
or equivalently, the subgraph of $G$ induced by $U$. 
For simplicity, we refer to both the position and the induced subgraph merely as the \emph{position $U$} or the \emph{DAG $U$}, omitting the arc set $A$ unless necessary. 

For a DAG $V$, the \emph{outcome $o(V)$ of $V$}, representing which player has a winning strategy in $V$, is formally defined by the following recursive formula.
\begin{definition}
\begin{align}
o(V) = \begin{cases}
\NN &\,\,\text{if}\,\,\PP \in \{o(V \setminus S) : S \subseteq L_0(V), S \neq \emptyset\},\\
\PP &\,\,\text{otherwise}.
\end{cases}
\end{align}
\end{definition}
A position $V$ is called an \emph{$\NN$-position} (resp.~\emph{$\PP$-position}) if $o(V) = \NN$ (resp.~$o(V) = \PP$).
Intuitively, an $\NN$-position represents a winning position for the next player, as there exists at least one legal move leading to a position that is losing for the opponent.
Conversely, a $\PP$-position represents a losing position for the next player, as every legal move leads to a winning position for the opponent.\footnote{In the literature on the Ungar games, \(\NN\)-positions and \(\PP\)-positions are referred to as \emph{Atniss win} and \emph{Eeta win}, respectively. In this paper, we adopt a conventional notation of combinatorial game theory.}

Let $d^+_V(x)$ and $d^-_V(x)$ denote the out-degree and in-degree of a node $x$ in a DAG $V$.
For $x, y \in V$, if there exists an arc from $x$ to $y$, then we write $x \to y$.
If $x \to y$ and $d^{-}_V(y) = 1$ (i.e., $y$ receives an arc only from $x$), then we write $x \twoheadrightarrow y$.
A \emph{path} on $V$ is a non-empty sequence $p = (x_1, x_2, \ldots, x_n)$ of nodes of $V$ with $x_1 \to x_2 \to \cdots \to x_n$,
whose length is defined as $n-1$ (i.e., the number of arcs through the path).
Let $l_V(x)$ denote the maximum length of all paths on $V$ ending with a node $x \in V$.
For a non-negative integer $k$, we define $L_k(V) \coloneqq \{x \in V : l_V(x) = k\}$.
We also define $l(V) \coloneqq \max_{x \in V} l_V(x)$ as the maximum length of a path on $V$.

In addition, we introduce the following notations for the sets of the nodes reachable from a given node or set of nodes.
\begin{definition}
\label{def:ideal}
For a DAG $V$, we introduce the following definitions (i) and (ii).
\begin{enumerate}[(i)]
\item $\Reach_{V}(x) \coloneqq \{x' \in V : x \rightsquigarrow x'\}$ and $\reach_{V}(x) \coloneqq \Reach_{V}(x) \setminus \{x\}$,
where $x \rightsquigarrow x'$ denotes that there exists a path from $x$ to $x'$.
\item For $U \subseteq V$, we define $\Reach_{V}(U) \coloneqq \bigcup_{u \in U} \Reach_{V}(u)$ and $\reach_{V}(U) \coloneqq \bigcup_{u \in U} \reach_{V}(u)$.
\end{enumerate}
\end{definition}
In the above definition, note that $x \rightsquigarrow x$ for any $x \in V$ because the sequence $(x)$ is a path on $V$ of length $0$.
Directly from the definition, we have
\begin{align}
\reach(\Reach(x)) &\subseteq \Reach(\Reach(x)) = \Reach(x), \label{eq:w1v090xsyy5q}\\
L_0(U) &= U \setminus \reach(U). \label{eq:kqut8anmwdmy}
\end{align}

In the above notations, the subscript $V$ indicating the underlying DAG is omitted when it is clear from the context.

A node $x$ of a DAG $V$ is said to be \emph{available in $V$} if $x \in L_0(V)$.
Equivalently, a legal move on $V$ is removing one or more available nodes in $V$.

If there exists an available node $x$ such that removing it does not create a new available node,
then it is seen that $o(V) = \NN$ from the so-called strategy-stealing argument.
If $V' \coloneqq V \setminus \{x\}$ is a $\PP$-position, then we immediately see that $o(V) = \NN$ since there is a legal move from $V$ to $V'$.
If $V'$ is an $\NN$-position, then $V'$ must have a legal move to some $\PP$-position $V'' \coloneqq V' \setminus S$ by removing some non-empty $S \subseteq L_0(V')$.
However, since our assumption ensures that $V'$ does not contain any nodes that became available as a result of the removal of $x$,
removing $\{x\} \cup S$ was a legal move from the very beginning position $V$ to reach $V'' = V \setminus (\{x\} \cup S)$, so that $o(V) = \NN$.
We formalize this argument as follows.

\begin{definition}
Let $V$ be a DAG. A node $x \in V$ is said to be \emph{dominated in $V$}
if $\reach(x) \subseteq \reach(V \setminus \{x\})$.
\end{definition}

\begin{lemma}
\label{thm:dominated}
If there exists an available dominated node in a DAG $V$, then $o(V) = \NN$.
\end{lemma}
\begin{proof}
Let $x \in L_0(V)$ be dominated.
Since removing $\{x\} \subseteq L_0(V)$ is a legal move, if $o(V \setminus \{x\}) = \PP$, then we are done.
Assume $o(V \setminus \{x\}) = \NN$.
Then there exists a non-empty $S \subseteq L_0(V \setminus \{x\})$ such that
\begin{align}
o((V\setminus\{x\}) \setminus S) = \PP. \label{eq:igk3rvk8v2or}
\end{align} 
We have
\begin{align}
\reach(V \setminus \{x\})
\eqlab{A}= \reach(V \setminus \{x\}) \cup \reach(x)
= \reach(V), \label{eq:pizlu51kv7kc}
\end{align}
where (A) follows since $x$ is dominated in $V$.
Hence,
\begin{align}
S
\subseteq L_0(V \setminus \{x\})
\eqlab{A}= (V \setminus \{x\}) \setminus \reach(V \setminus \{x\})
\eqlab{B}= (V \setminus \{x\}) \setminus \reach(V)
\subseteq V \setminus \reach(V)
\eqlab{C}= L_0(V), \label{eq:fm1vaz90ebn6}
\end{align}
where
(A) follows from \eqref{eq:kqut8anmwdmy},
(B) follows from \eqref{eq:pizlu51kv7kc},
and (C) follows from \eqref{eq:kqut8anmwdmy}.
Combining this and $\{x\} \subseteq L_0(V)$,
we obtain $S \cup \{x\} \subseteq L_0(V)$, so that
removing $S \cup \{x\}$ is a legal move from the position $V$, and thus we conclude $o(V) = \NN$ by \eqref{eq:igk3rvk8v2or}.
\end{proof}

Note that $x$ with out-degree $0$ is trivially dominated.
In particular, an isolated node is an available dominated node,
which yields the following corollary.

\begin{corollary}
\label{cor:isolated}
If a DAG $V$ has an isolated node, then $o(V) = \NN$.
\end{corollary}

A node $x$ of a DAG $V$ is called a \emph{leaf of $V$} if $d^+(x) = 0$ and $x \twoheadleftarrow \hat{x}$ for some $\hat{x}$.
If $V$ has a leaf $x$ with $x \twoheadleftarrow \hat{x}$, then
the removal of $\hat{x}$ leaves $x$ isolated and results in an $\NN$-position by Corollary \ref{cor:isolated},
forcing a loss for the player making such a move.
Therefore, even if we prohibit removing $\hat{x}$, the outcome remains the same.
In that case, the nodes of $\Reach(\hat{x})$ would no longer be removed,
and thus the outcome remains the same even if we treat the nodes of $R(\hat{x})$ as nonexistent.
We formalize this argument as the following lemma.

\begin{lemma}
\label{thm:ignore}
Let $V$ be a DAG. For any leaf $x \in V$, we have $o(V) = o(V \setminus \Reach(\hat{x}))$, where $\hat{x} \twoheadrightarrow x$.
\end{lemma}

\begin{proof}
We proceed by induction on $|V|$.

We first claim that for any non-empty $S \subseteq L_0(V)$, we have
\begin{align}
o(V \setminus S) = 
\begin{cases}
\NN &\,\,\text{if}\,\,\hat{x} \in S,\\
o((V \setminus S) \setminus \Reach(\hat{x})) &\,\,\text{if}\,\,\hat{x} \not\in S\\
\end{cases}
\label{eq:0k9jvt8krszy}
\end{align}
as follows.

\begin{itemize}
\item The case $\hat{x} \in S$:
Since $\hat{x} \twoheadrightarrow x$,
the node $x$ is isolated in $V \setminus S$, so that $o(V \setminus S) = \NN$ from Corollary \ref{cor:isolated}.

\item The case $\hat{x} \not\in S$: We may apply the induction hypothesis to $V \setminus S$ to obtain $o(V \setminus S) =  o((V \setminus S) \setminus \Reach(\hat{x}))$.
\end{itemize}
Note that in both cases, the node $x$ indeed remains in $V \setminus S$ because
\begin{align}
\{x\} \eqlab{A}\subseteq \reach(\hat{x}) \subseteq \reach(V) \eqlab{B}= V \setminus L_0(V) \subseteq V \setminus S,
\label{eq:fkaij8rvkd2c}
\end{align}
where
(A) follows since $\hat{x} \to x$,
and (B) follows from \eqref{eq:kqut8anmwdmy}. 

Finally, we have
\begin{align}
o(V) = \NN
&\iff \PP \in \{o(V \setminus S) : S \subseteq L_0(V), S \neq \emptyset\}\\
&\iff \PP \in \{o(V \setminus S) : S \subseteq L_0(V), S \owns \hat{x}\} \cup \{o(V \setminus S) : S \subseteq L_0(V) \setminus \{\hat{x}\}, S \neq \emptyset\}\\
&\eqlab{A}\iff \PP \in \{o(V \setminus S) : S \subseteq L_0(V) \setminus \{\hat{x}\}, S \neq \emptyset\}\\
&\eqlab{B}\iff \PP \in \{o((V \setminus \Reach(\hat{x})) \setminus S) : S \subseteq L_0(V) \setminus \{\hat{x}\}, S \neq \emptyset\}\\
&\eqlab{C}\iff \PP \in \{o((V \setminus \Reach(\hat{x})) \setminus S) : S \subseteq L_0(V \setminus \Reach(\hat{x})), S \neq \emptyset\}\\
&\iff o(V \setminus \Reach(\hat{x})) = \NN
\end{align}
as desired, where
(A) follows because $\{o(V \setminus S) : S \subseteq L_0(V), S \owns \hat{x}\} \subseteq \{\NN\}$ by the first case of \eqref{eq:0k9jvt8krszy},
(B) follows from the second case of \eqref{eq:0k9jvt8krszy},
and (C) is justified as follows:
\begin{align}
L_0(V) \setminus \{\hat{x}\}
&\eqlab{A}= V \setminus (\reach(V) \cup \{\hat{x}\})\\
&\eqlab{B}= V \setminus (\reach(V) \cup \reach(\hat{x}) \cup \{\hat{x}\})\\
&= V \setminus (\reach(V) \cup \Reach(\hat{x}))\\
&= V \setminus \left(\reach(V \setminus \Reach(\hat{x})) \cup \reach(\Reach(\hat{x})) \cup \Reach(\hat{x})\right)\\
&\eqlab{C}= V \setminus \left(\reach(V \setminus \Reach(\hat{x})) \cup \Reach(\hat{x})\right)\\
&= (V \setminus \Reach(\hat{x})) \setminus \reach(V \setminus \Reach(\hat{x}))\\
&\eqlab{D}= L_0(V \setminus \Reach(\hat{x})),
\end{align}
where
(A) follows from \eqref{eq:kqut8anmwdmy},
(B) follows since $\reach(\hat{x}) \subseteq \reach(V)$,
(C) follows from \eqref{eq:w1v090xsyy5q},
and (D) follows from \eqref{eq:kqut8anmwdmy}.
\end{proof}

\section{Main Results}
\label{sec:main}

In this section, we establish our main results on the computational complexity of the Ungar game on a DAG.
To prove the hardness results, we construct DAGs for the Ungar game that simulate instances of the source problems.
For the construction of DAGs in the proofs, we introduce a gadget denoted by $\mathrm{SIM}(m, n)$, which enforces a move in the Ungar game on a DAG to simultaneously (hence named $\mathrm{SIM}$) remove exactly one node from a set $\{x_1, \ldots, x_m\}$ of nodes and exactly one node from a set $\{y_1, \ldots, y_n\}$ of nodes, as follows.

For a positive integer $n$, let $[n] \coloneqq \{1, 2, \ldots, n\}$.
For positive integers $m$ and $n$, the gadget $\mathrm{SIM}(m, n)$ is a subgraph $(\hat{V}, \hat{A})$ with
\begin{align}
\hat{V} &\coloneqq \bigcup_{i \in [m]}\{x_i, \alpha_i\} \cup
\bigcup_{j \in [n]} \{y_j, \gamma_j\} \cup
\bigcup_{i, j \in [m], i < j} \{p_{i, j}\} \cup
\bigcup_{i, j \in [n], i < j} \{q_{i, j}\} \cup
\bigcup_{i \in [m], j \in [n]}\{\beta_{i, j}, \delta_{j, i}\},\\
\hat{A} &\coloneqq \bigcup_{i \in [m]} \{x_i \to \alpha_i\} \cup
\bigcup_{j \in [n]} \{y_j \to \gamma_j\}\\
&\quad \cup \bigcup_{i, j \in [m], i < j}\{x_i \to p_{i, j} \leftarrow x_j\} \cup \bigcup_{i, j \in [n], i < j} \{y_i \to q_{i, j} \leftarrow y_j\}\\
&\quad \cup \bigcup_{i \in [m], j \in [n]}\{\alpha_i \to \beta_{i, j}, \gamma_j \to \delta_{j, i}, x_i \to \delta_{j, i}, y_j \to \beta_{i, j}\},
\end{align}
where $x \to y$ denotes an arc from $x$ to $y$, and $x \to y \leftarrow z$ denotes the two arcs $x \to y$ and $z \to y$.
When this gadget is embedded into a DAG $(V, A)$ as a subgraph, only the nodes $x_1, x_2, \ldots, x_m, y_1, y_2, \ldots, y_n$ may serve as endpoints for other arcs of $(V, A)$ outside the gadget.
All other nodes in $\hat{V}$ cannot be connected to any arcs of $(V, A)$ other than those defined above for the gadget to work correctly.

For example, Figures~\ref{fig:gad-1-1}, \ref{fig:gad-1-mul}, and \ref{fig:gad-mul-mul} illustrate
$\mathrm{SIM}(1, 1), \mathrm{SIM}(3, 1)$, and $\mathrm{SIM}(3, 2)$, respectively.
In these figures, nodes marked as black circles may be endpoints of any arc other than those shown in the figure, whereas those marked as white circles may not.
\begin{figure}[htbp]
\centering
\includegraphics[keepaspectratio,scale=1.5]{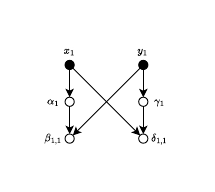}
\caption{A gadget $\mathrm{SIM}(1, 1)$}
\label{fig:gad-1-1}
\end{figure}
\begin{figure}[htbp]
\centering
\includegraphics[keepaspectratio,scale=1]{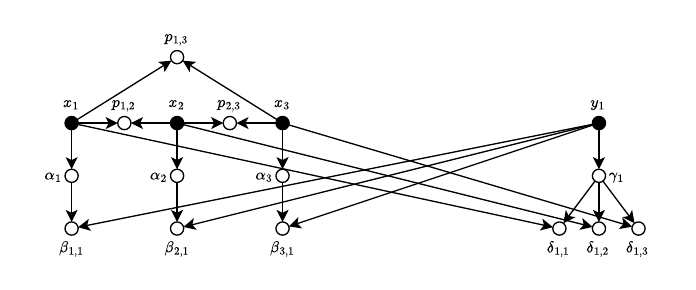}
\caption{A gadget $\mathrm{SIM}(3, 1)$}
\label{fig:gad-1-mul}
\end{figure}
\begin{figure}[htbp]
\centering
\includegraphics[keepaspectratio,scale=1]{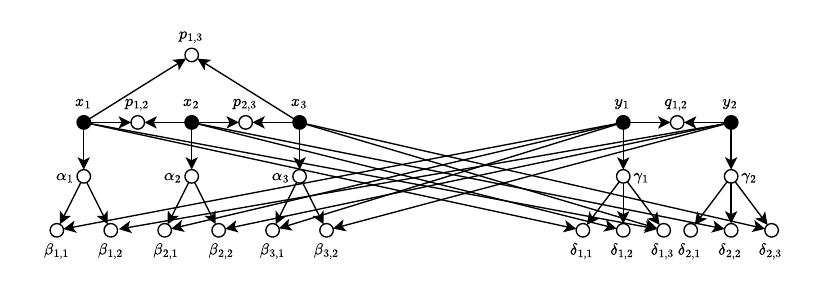}
\caption{A gadget $\mathrm{SIM}(3, 2)$}
\label{fig:gad-mul-mul}
\end{figure}

A $\mathrm{SIM}(m, n)$ gadget enforces a simultaneous selection of exactly one node from each of $X \coloneqq \{x_1, \ldots, x_m\}$ and $Y \coloneqq \{y_1, \ldots, y_n\}$ by ensuring that any violation immediately hands a winning position to the opponent, as follows.
The nodes $p_{i, j}$ (resp.~$q_{i, j}$) guarantee that at most one node of $X$ and at most one node of $Y$ are removed.
Indeed, removing two or more nodes from $X$ (resp.~$Y$) leaves some $p_{i, j}$ (resp.~$q_{i, j}$) isolated, handing a winning position to the opponent by Corollary \ref{cor:isolated}.
Moreover, if a player removes a node of $X$, they are forced to remove a node of $Y$ in the same turn, and vice versa because
if a player removes some $x_i$ but leaves all $y_1, y_2, \ldots, y_n$ unremoved, then $\alpha_i$ becomes an available dominated node, turning the resulting graph into an $\NN$-position by Lemma \ref{thm:dominated}.

Conversely, suppose both some $x_i$ and some $y_j$ are removed in a single move,
where we may let $(i, j) = (1, 1)$ without loss of generality.
Then the remaining parts in the gadget can be safely ignored:
by the removal, the nodes $\beta_{1, 1}, \beta_{2, 1}, \ldots, \beta_{m, 1}, \delta_{1, 1}, \delta_{2, 1}, \ldots, \delta_{n, 1}, p_{1, 2}, p_{1, 3}, \ldots, \allowbreak p_{1, m}, q_{1, 2}, q_{1, 3}, \ldots, q_{1, n}$ become leaves,
so that all the remaining parts can be extinguished by Lemma \ref{thm:ignore} without changing the outcome
(equivalently, making any move on the remaining part leaves an isolated node, handing a winning position to the opponent).

\subsection{The \textsf{PSPACE}-completeness}
\label{subsec:pspace}

We show that the problem of determining the outcome of the Ungar game on a given DAG is \textsf{PSPACE}-complete.
The hardness is shown by a polynomial-time reduction from \textsc{Node-Kayles}, which is a combinatorial game played on an undirected graph as follows.
\begin{quote}
The game is played on an undirected graph $\Gamma = (V_{\Gamma}, E_{\Gamma})$.
In each turn, a player selects a node of ${\Gamma}$.
Once a node $x$ has been selected, neither that node $x$ nor any of its adjacent nodes can be selected no longer.
The player who cannot make a legal move loses.
\end{quote}
The problem of determining the outcome of \textsc{Node-Kayles} on a given graph $\Gamma$ is known to be \textsf{PSPACE}-complete \cite{Sch78}.

One of our main results is formally stated as follows, reformulating the statement in Section~\ref{sec:intro} to align with the formulation by a DAG.
\begin{theorem}
\label{thm:pspace}
The problem of determining the outcome of the Ungar game on a given DAG $V$ is \textsf{PSPACE}-complete even if
we restrict $V$ to the cases where $l(V) \leq 3$ and the underlying graph of $V$ is bipartite.
\end{theorem}
\begin{proof}
Since every play on the Ungar game starting from the initial position $V$ terminates within $|V|$ moves,
the outcome can be computed in polynomial space by recursive and exhaustive traversal of all subpositions,
so that our problem belongs to \textsf{PSPACE}.

We prove the \textsf{PSPACE}-hardness by showing a polynomial-time reduction from \textsc{Node-Kayles}.
Suppose that we are given an undirected graph $\Gamma = (V_{\Gamma}, E_{\Gamma})$ for \textsc{Node-Kayles},
where $V_{\Gamma} = \{v_1, v_2, \ldots, v_n\}$. We construct a DAG $G = (V_G, A_G)$ for the Ungar game as follows:
\begin{align}
V_G &\coloneqq \bigcup_{t, i \in [n]}\{v_i^{(t)}, w_i^{(t)}, x_i^{(t)}, y_i^{(t)}\}\\
&\quad \cup \bigcup_{t \in [n]}\{\alpha^{(t)}, \beta^{(t)},
\gamma^{(t)}, \delta^{(t)}, \epsilon^{(t)}, \zeta^{(t)}, \eta^{(t)}, \theta^{(t)},
\iota^{(t)}, \kappa^{(t)}, \lambda^{(t)}, \mu^{(t)}, \nu^{(t)}, \xi^{(t)}\}\\
&\quad \cup \bigcup_{t, i, j \in [n], i < j}\{p_{i, j}^{(t)}\}
\cup \bigcup_{t, u, i\in [n], t < u}\{q_i^{(t, u)}\}
\cup \bigcup_{t, u \in [n], t < u, \{i, j\} \in E_{\Gamma}} \{r_{i, j}^{(t, u)}, r_{j, i}^{(t, u)}\},\\
A_G &\coloneqq \bigcup_{t, i \in [n]}\{v_i^{(t)} \to w_i^{(t)} \to x_i^{(t)}, v_i^{(t)} \to y_i^{(t)},
\alpha^{(t)} \to x_i^{(t)}, \beta^{(t)} \to y_i^{(t)}\}\\
&\quad \cup \bigcup_{t \in [n]}\{\alpha^{(t)} \to \beta^{(t)}, \alpha^{(t)} \to \gamma^{(t)} \to \delta^{(t)} \to \epsilon^{(t)}, \gamma^{(t)} \to \theta^{(t)}, \zeta^{(t)} \to \eta^{(t)} \to \theta^{(t)}, \zeta^{(t)} \to \epsilon^{(t)},\\
&\qquad\zeta^{(t)} \to \iota^{(t)} \to \kappa^{(t)} \to \lambda^{(t)}, \iota^{(t)} \to \xi^{(t)}, \mu^{(t)} \to \nu^{(t)} \to \xi^{(t)}, \mu^{(t)} \to \lambda^{(t)}\}\\
&\quad\cup \bigcup_{t \in [n-1], i \in [n]} \{\mu^{(t)} \to v_i^{(t+1)}\}\\
&\quad\cup \bigcup_{t, i, j \in [n], i < j} \{v_i^{(t)} \to p_{i, j}^{(t)} \leftarrow v_j^{(t)}\}\\
&\quad\cup \bigcup_{t, u, i \in [n], t < u} \{v_i^{(t)} \to q_{i}^{(t, u)} \leftarrow v_i^{(u)}\}\\
&\quad\cup \bigcup_{t, u \in [n], t < u, \{i, j\} \in E_{\Gamma}} \{v_i^{(t)} \to r_{i, j}^{(t, u)} \leftarrow v_j^{(u)}, v_i^{(u)} \to r_{j, i}^{(t, u)} \leftarrow v_j^{(t)}\}.
\end{align}

\begin{figure}[htbp]
\centering
\includegraphics[keepaspectratio,scale=0.83]{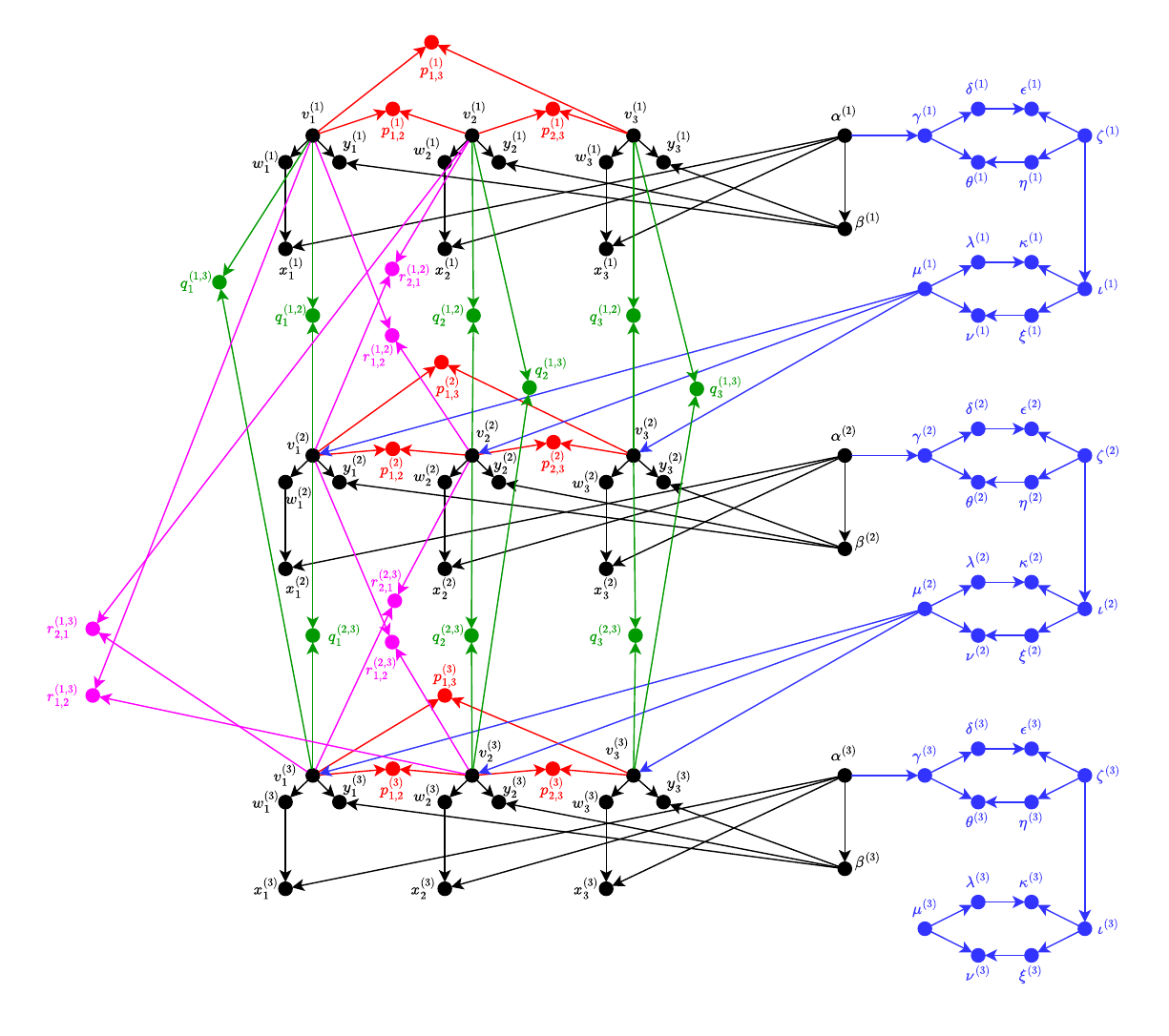}
\caption{The DAG corresponding to a graph $\Gamma = (V_{\Gamma}, E_{\Gamma})$,
where $V_{\Gamma} = \{v_1, v_2, v_3\}, E_{\Gamma} = \{\{v_1, v_2\}\}$}
\label{fig:pspace}
\end{figure}

For example, Fig.~\ref{fig:pspace} shows $G$ corresponding to $V_{\Gamma} = \{v_1, v_2, v_3\}, E_{\Gamma} = \{\{v_1, v_2\}\}$.

We claim that $o(V) = \NN$ if and only if the first player has a winning strategy in \textsc{Node-Kayles} on $\Gamma$.
Intuitively, the DAG $G$ simulates a play on \textsc{Node-Kayles} on $\Gamma$.
Removing a node $v_i^{(t)}$ corresponds to selecting the node $v_i \in V_{\Gamma}$ at turn $t$ in \textsc{Node-Kayles}.

Consider a play on $G$ by two players, say Alice and Bob, from initial position $V$, where Alice plays first.
At first, the available nodes are $\{v_1^{(1)}, v_2^{(1)}, \ldots, v_n^{(1)}, \alpha^{(1)}, \alpha^{(2)}, \ldots, \alpha^{(n)}, \zeta^{(1)}, \zeta^{(2)}, \allowbreak \ldots, \zeta^{(n)}, \mu^{(1)}, \mu^{(2)}, \ldots, \mu^{(n)}\}$.
Note that the subgraph induced by $\{\gamma^{(t)}, \delta^{(t)}, \epsilon^{(t)}, \zeta^{(t)}, \eta^{(t)}, \theta^{(t)}\}$
(resp.~$\{\iota^{(t)}, \kappa^{(t)}, \lambda^{(t)}, \mu^{(t)}, \nu^{(t)}, \xi^{(t)}\}$)
forms a $\mathrm{SIM}(1, 1)$, so that
$\gamma^{(t)}$ and $\zeta^{(t)}$ (resp.~$\iota^{(t)}$ and $\mu^{(t)}$) must be removed at the same time.
Hence, Alice cannot remove any of the available nodes $\zeta^{(t)}$ (resp.~$\mu^{(t)}$) at this time because they must be removed 
at the same turn as the removal of $\gamma^{(t)}$ (resp.~$\iota^{(t)}$), which cannot be removed legally at this time.
Further, for every $t \in [n]$, the subgraph induced by 
\begin{align}
\{\alpha^{(t)}, \beta^{(t)}\}
\cup \bigcup_{i \in [n]} \{v_i^{(t)}, w_i^{(t)}, x_i^{(t)}, y_i^{(t)}\}
\cup \bigcup_{i, j \in [n], i < j} \{p_{i, j}^{(t)}\}
\end{align}
forms a $\mathrm{SIM}(n, 1)$ and requires that removing $\alpha^{(t)}$ and removing exactly one of $\{v_1^{(t)}, v_2^{(t)}, \ldots, \allowbreak v_n^{(t)}\}$ must be done simultaneously.
However, $\{v_1^{(t)}, v_2^{(t)}, \ldots, v_n^{(t)}\}$ for $t \geq 2$ are not available,
thus the nodes $\alpha^{(2)}, \alpha^{(3)}, \ldots, \alpha^{(n)}$ cannot be removed.
Consequently, the first move must be in the form $S = \{\alpha^{(1)}, v_i^{(1)}\}$ for some $i \in [n]$,
which corresponds to selecting $v_i \in V_{\Gamma}$ at the first turn in \textsc{Node-Kayles}.

After Alice's first move, the node $\beta^{(1)}$ becomes available for Bob by Alice's removal of $\alpha^{(1)}$.
Removing $\{\beta^{(1)}, \zeta^{(1)}\}$ is Bob's only move that does not immediately result in a loss for Bob,
since the other available nodes are safely ignored as discussed at the beginning of this section.
Similarly, Alice's next move is limited to removing $\{\iota^{(1)}, \mu^{(1)}\}$,
which makes the nodes $v_1^{(2)}, v_2^{(2)}, \ldots, v_n^{(2)}$ available for the second turn of \textsc{Node-Kayles} played by Bob.
In Fig.~\ref{fig:pspace}, the corresponding parts are drawn in blue.
The game continues similarly for subsequent turns $t = 2, 3, \ldots$.
Since any play on \textsc{Node-Kayles} terminates within $n$ moves, the constructed DAG $G$ contains a sufficient number of nodes to simulate the entire play.

During the simulation, a move on $G$ corresponding to an illegal move on $\Gamma$ leaves an isolated node, forcing a loss for the player making such a move, as follows.
The nodes $q_i^{(t, u)}$ ensure that if $v_i$ was chosen in a previous turn $t$, choosing the same node $v_i$ at a later turn $u > t$ leaves $q_i^{(t, u)}$ isolated (in Fig.~\ref{fig:pspace}, the corresponding parts are drawn in green).
Also, the nodes $r_{i, j}^{(t, u)}$ and $r_{j, i}^{(t, u)}$ ensure that if $\{v_i, v_j\} \in E_{\Gamma}$ and $v_i$ was chosen at turn $t$, then $v_j$ cannot be selected at any subsequent turn $u > t$, reflecting the constraint of \textsc{Node-Kayles} that two adjacent nodes of $\Gamma$ cannot both be chosen (in Fig.~\ref{fig:pspace}, the corresponding parts are drawn in purple).

Note that remnants of the $\mathrm{SIM}$ gadgets generated during the simulation can be safely ignored by Lemma~\ref{thm:ignore}, as noted when introducing the gadget.

By the above construction, any play on $G$ that deviates from legal moves on $\Gamma$ leads directly to an $\NN$-position for the opponent.
Therefore, the first player has a winning strategy in \textsc{Node-Kayles} on $\Gamma$ if and only if $o(V) = \NN$.

The number of nodes and edges of $G$ is polynomial in $|V_{\Gamma}| + |E_{\Gamma}|$, 
so that the reduction is performed in polynomial time.
Finally, it is directly seen that $l(V) \le 3$ and that the underlying undirected graph of $G$ is bipartite.
This completes the proof.
\end{proof}

\subsection{The \textsf{NP}-completeness for $l(V) \leq 2$}
\label{subsec:np}

We show that determining the outcome of the Ungar game on a given DAG is \textsf{NP}-complete when $l(V) \le 2$.
The hardness is shown via a polynomial-time reduction from 3-SAT.
Recall that an instance of 3-SAT consists of a Boolean formula $F$ in 3-conjunctive normal form (3-CNF) over variables $X_1, X_2, \ldots, X_m$, and the problem asks whether there exists a truth assignment that satisfies $F$,
which is known to be \textsf{NP}-complete \cite{GJ79}.

Our second main result is stated as follows, again reformulating the statement in Section~\ref{sec:intro} to align with the formulation by a DAG.
\begin{theorem}
\label{thm:np}
The problem of determining the outcome of the Ungar game on a given DAG $V$ with $l(V) \leq 2$ is \textsf{NP}-complete.
\end{theorem}

\begin{proof}
(Proof of belonging to \textsf{NP})
As long as there exists a leaf in $V$, we may repeat the reduction by Lemma \ref{thm:ignore} without changing the outcome.
If the resulting $V$ has an available dominated node,
then we immediately conclude that the outcome is $\NN$ by Lemma \ref{thm:dominated}.
Since these processes can be performed in polynomial time,
we may assume that the given DAG $V$ has no leaves or available dominated nodes.

Let $S \subseteq L_0(V)$ be said to be \emph{safe} if $V \setminus S$ has no available nodes dominated in $V \setminus S$.
Let $S_{\max} \subseteq L_0(V)$ be a safe set with the maximum cardinality among all safe $S \subseteq L_0(V)$.
Note that such a set does indeed exist because the empty set is safe by the assumption that $V$ has no available dominated nodes.
We claim that $o(V \setminus S_{\max}) = \PP$ below.

To prove that, we choose a non-empty $S \subseteq L_0(V \setminus S_{\max})$ arbitrarily and confirm $o((V \setminus S_{\max}) \setminus S) = \NN$.
Since $S \subseteq L_0(V \setminus S_{\max}) \subseteq L_0(V) \cup L_1(V)$, there are two possible cases as follows.
\begin{itemize}
\item The case $S \cap L_1(V) \neq \emptyset$:
Let $x \in S \cap L_1(V)$.
Since $S_{\max}$ is safe, the subgraph $V \setminus S_{\max}$ has no available nodes dominated in $V \setminus S_{\max}$.
Hence, there exists $y \in V \setminus S_{\max}$ such that $y \in \reach(x)$ and $y \not\in \reach((V \setminus S_{\max}) \setminus \{x\})$, which is equivalent to $y \twoheadleftarrow x$ in $V \setminus S_{\max}$. 
Since $y \leftarrow x \in L_1(V)$ and $l(V) \leq 2$, the node $y$ is a leaf of $V \setminus S_{\max}$.
Therefore, removing $S$ makes $y$ isolated and results in $o((V \setminus S_{\max}) \setminus S) = \NN$ by Corollary \ref{cor:isolated}.
\item The case $S \subseteq L_0(V)$:
Removing $S$ leaves an available dominated node because $S_{\max} \cup S \subseteq L_0(V)$ has a greater cardinality than $S_{\max}$ and thus is not safe by the definition of $S_{\max}$ (i.e., the maximality in cardinality).
Therefore, $o((V \setminus S_{\max}) \setminus S) = \NN$ by Lemma \ref{thm:dominated}.
\end{itemize}
Consequently, we conclude that $o(V \setminus S_{\max}) = \PP$.
This shows that $o(V) = \NN$ if and only if $S_{\max} \neq \emptyset$.
Further, $S_{\max} \neq \emptyset$ is equivalent to the existence of a non-empty safe $S \subseteq L_0(V)$ (not necessarily with the maximum cardinality).
Hence, the problem of determining the existence of a non-empty safe $S$ belongs to \textsf{NP} because
it can be confirmed in polynomial time whether a given $S$ is safe or not.
Therefore, the problem of determining $o(V)$ is also in \textsf{NP}.

(Proof of the \textsf{NP}-hardness)
We show a polynomial-time reduction from 3-SAT.
Suppose that we are given a Boolean formula $F(X_1, X_2, \ldots, X_m) = \bigwedge_{i = 1}^n (A_i \lor B_i \lor C_i)$ in 3-CNF, where $A_i, B_i, C_i$ are literals (i.e., either $X_j$ or $\bar{X}_j$ for some $j \in [m]$) for $i \in [n]$.
We may suppose $m \geq 2$ for the reduction.

We construct a DAG $G = (V, A)$ corresponding to $F$ as follows:
\begin{align}
V &\coloneqq \bigcup_{i \in [m]}\{X_i, \bar{X}_i, \chi_i\}
\cup \bigcup_{i, j \in [m], i \neq j}\{\alpha_{i, j}, \beta_{i, j}, \gamma_{i, j},
\bar{\alpha}_{i, j}, \bar{\beta}_{i, j}, \bar{\gamma}_{i, j}\}
\cup \bigcup_{i \in [n]}\{y_i\},\\
A &\coloneqq \bigcup_{i \in [m]}\{X_i \to \chi_i \leftarrow \bar{X}_i\}\\
&\cup \bigcup_{i, j \in [m], i \neq j}\{X_i \to \alpha_{i, j}, \alpha_{i, j} \to \beta_{i, j}, \alpha_{i, j} \to \gamma_{i, j}, X_i \to \beta_{j, i}, X_i \to \bar{\beta}_{j, i}\}\\
&\cup \bigcup_{i, j \in [m], i \neq j}\{\bar{X}_i \to \bar{\alpha}_{i, j}, \bar{\alpha}_{i, j} \to \bar{\beta}_{i, j}, \bar{\alpha}_{i, j} \to \bar{\gamma}_{i, j}, \bar{X}_i \to \gamma_{j, i}, \bar{X}_i \to \bar{\gamma}_{j, i}\}\\
&\cup \bigcup_{i \in [n]}\{\bar{A}_i \to y_i, \bar{B}_i \to y_i, \bar{C}_i \to y_i\} \label{eq:du3hosocjrfu}.
\end{align}

Note that each literal $X_i$ (or $\bar{X}_i$) also serves as a node in $V$.
The arcs $\bar{A}_i \to y_i, \bar{B}_i \to y_i, \bar{C}_i \to y_i$ in the last term mean that $y_i$ receives directed arcs from the negations of literals in the clause $(A_i \lor B_i \lor C_i)$. 

\begin{figure}[htbp]
\centering
\includegraphics[keepaspectratio,scale=0.8]{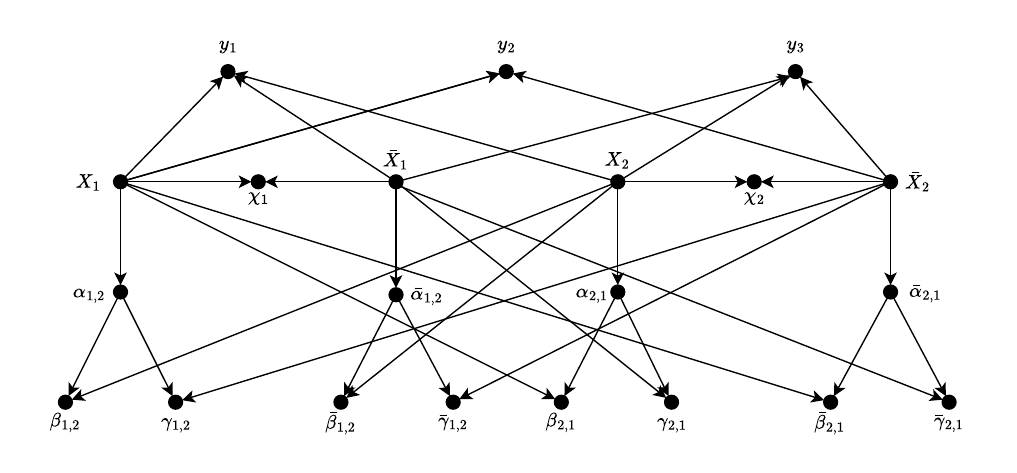}
\caption{The DAG corresponding to $F = (X_1 \lor \bar{X}_1 \lor \bar{X}_2) \land (\bar{X}_1 \lor \bar{X}_1 \lor X_2) \land (X_1 \lor X_2 \lor \bar{X}_2)$}
\label{fig:np-2}
\end{figure}
\begin{figure}[htbp]
\centering
\includegraphics[keepaspectratio,scale=0.59]{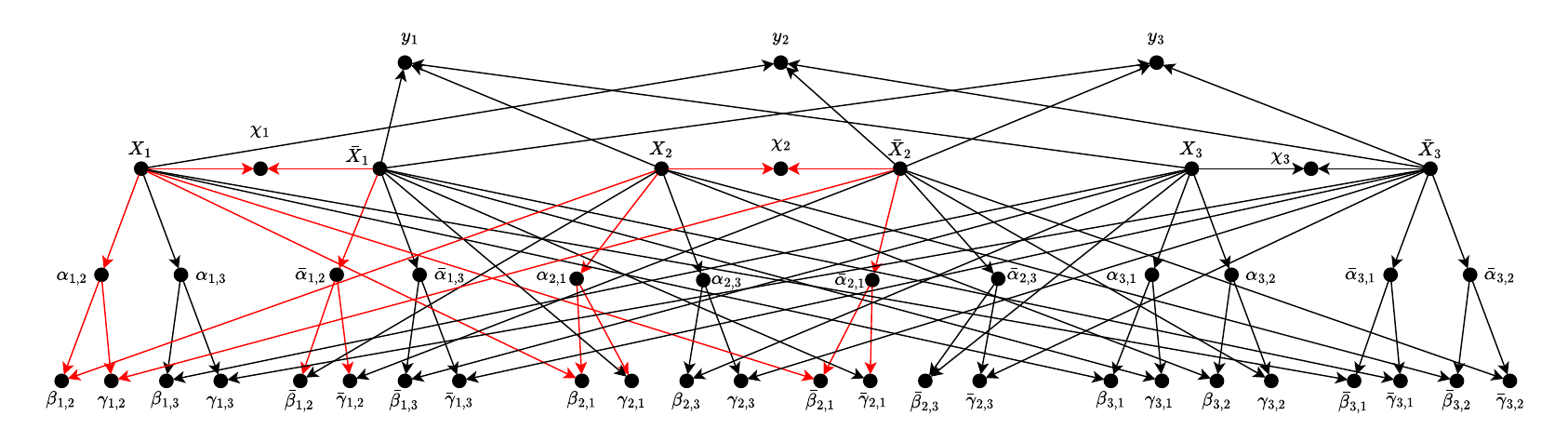}
\caption{The DAG corresponding to $F = (X_1 \lor \bar{X}_2 \lor \bar{X}_3) \land (\bar{X}_1 \lor X_2 \lor X_3) \land (X_1 \lor X_2 \lor X_3)$}
\label{fig:np-3}
\end{figure}

For example, Figures~\ref{fig:np-2} and \ref{fig:np-3} show $G$ corresponding to $F = (X_1 \lor \bar{X}_1 \lor \bar{X}_2) \land (\bar{X}_1 \lor \bar{X}_1 \lor X_2) \land (X_1 \lor X_2 \lor \bar{X}_2)$ and $F = (X_1 \lor \bar{X}_2 \lor \bar{X}_3) \land (\bar{X}_1 \lor X_2 \lor X_3) \land (X_1 \lor X_2 \lor X_3)$, respectively.

We claim that $o(V) = \NN$ if and only if the formula $F$ is satisfiable.
Intuitively, the first player's choice of a subset $S \subseteq L_0(V) = \bigcup_{i \in [m]}\{X_i, \bar{X}_i\}$ corresponds to choosing a truth assignment for the variables $X_1, X_2, \ldots, X_m$.
Including $X_i$ in $S$ sets $X_i$ to \textsf{true}, while including $\bar{X}_i$ sets $X_i$ to \textsf{false}. 

We first show that the first player can only choose a set $S$ that contains exactly one literal node for each variable.
For each $i, j \in [m]$ with $i \neq j$, the subgraph induced by
\begin{align}
V_{i, j} &\coloneqq
\{X_i, \alpha_{i, j}, \beta_{i, j}, \gamma_{i, j}, \bar{X}_i, \bar{\alpha}_{i, j}, \bar{\beta}_{i, j}, \bar{\gamma}_{i, j}, \chi_i\} \cup
\{X_j, \alpha_{j, i}, \beta_{j, i}, \gamma_{j, i}, \bar{X}_j, \bar{\alpha}_{j, i}, \bar{\beta}_{j, i}, \bar{\gamma}_{j, i}, \chi_j\}
\end{align}
forms a $\mathrm{SIM}(2, 2)$. 
For example, in Fig.~\ref{fig:np-3}, the part colored in red forms a $\mathrm{SIM}(2, 2)$ for $(i, j) = (1, 2)$.
Since $S$ is non-empty, we have $X_k \in S$ for some $k \in [m]$.
For every $j \in [m] \setminus \{k\}$, the subgraph $V_{k, j}$ consisting of a $\mathrm{SIM}(2, 2)$ forces $S$ to contain exactly one of $\{X_k, \bar{X}_k\}$ and exactly one of $\{X_j, \bar{X}_j\}$.
This forces $S$ to choose exactly one literal for every variable $X_i$.

Accordingly, in what follows, we assume that $|S \cap \{X_i, \bar{X}_i\}| = 1$ for every $i \in [m]$.
We define \emph{the assignment corresponding to $S \subseteq L_0(V)$} as the truth assignment for $X_1, X_2, \ldots, X_m$ such that $X_i$ is assigned \textsf{true} if and only if $X_i \in S$.
We claim that $o(V \setminus S) = \PP$ if and only if the assignment corresponding to $S$ satisfies the formula $F$.

Note that the node $y_i$ receives arcs from $\bar{A}_i, \bar{B}_i, \bar{C}_i$ (i.e., the nodes corresponding to the negations of the clause's literals).
Therefore, if the assignment corresponding to $S$ does not satisfy $F$,
then all three literals $A_i, B_i, C_i$ were evaluated as \textsf{false} for some clause $(A_i \lor B_i \lor C_i)$,
so that $\bar{A}_i, \bar{B}_i, \bar{C}_i$ would all be removed by $S$, leaving $y_i$ isolated and resulting in $o(V \setminus S) = \NN$.

Conversely, when $S$ corresponds to a satisfying assignment, applying Lemma~\ref{thm:ignore} repeatedly allows us to reduce $V \setminus S$ to the empty set, establishing $o(V \setminus S) = \PP$.
Specifically, the remnants of the $\mathrm{SIM}$ gadgets can be safely eliminated as noted when introducing the gadget.
In addition, for every $i \in [n]$, because the assignment corresponding to $S$ satisfies $F$, there exists an arc from the negation of a satisfied literal in the $i$-th clause to $y_i$ in $V \setminus S$.
Suppose $y_i \leftarrow \bar{X}_1$ in $V \setminus S$, for example.
Then $X_1$ must be contained in $S$ and removed, so that $\bar{X}_1$ has an arc to a leaf $\chi_1$ in $V \setminus S$,
which leads to that $y_i \in \Reach_{V \setminus S}(\bar{X}_1)$ is removed via Lemma~\ref{thm:ignore}.

Consequently, the first player has a winning move $S$ if and only if there exists a set $S$ containing exactly one literal for each variable such that the corresponding assignment satisfies $F$, which is equivalent to $F$ being satisfiable.

The number of nodes and edges of $G$ is polynomial in $m + n$,  so that the reduction is performed in polynomial time.
Also, it is directly seen that $l(V) \le 2$.
This completes the proof.
\end{proof}

\section*{Acknowledgements}
This work was supported in part by JSPS KAKENHI Grant Number JP26K21166.

\bibliographystyle{plain}
\bibliography{draft}

\end{document}